\documentclass[12pt]{amsart}
\usepackage{amsmath,amssymb,amsbsy,amsfonts,amsthm,latexsym,mathabx,
            amsopn,amstext,amsxtra,euscript,amscd,stmaryrd,mathrsfs,
            cite,array,mathtools,enumerate}

\usepackage{url}
\usepackage[colorlinks,linkcolor=blue,anchorcolor=blue,citecolor=blue,backref=page]{hyperref}

\usepackage[norefs,nocites]{refcheck}
\usepackage{color}
\usepackage{float}

\hypersetup{breaklinks=true}

\usepackage[english]{babel}

\usepackage{mathtools}
\usepackage{todonotes}

\usepackage{enumitem}

\usepackage{mathtools}
\usepackage{todonotes}
\usepackage{url}
\usepackage[colorlinks,linkcolor=blue,anchorcolor=blue,citecolor=blue,backref=page]{hyperref}

\begin{document}

\newtheorem{theorem}{Theorem}
\newtheorem{lemma}[theorem]{Lemma}
\newtheorem{claim}[theorem]{Claim}
\newtheorem{cor}[theorem]{Corollary}
\newtheorem{prop}[theorem]{Proposition}
\newtheorem{definition}{Definition}
\newtheorem{question}[theorem]{Open Question}
\newtheorem{example}[theorem]{Example}
\newtheorem{remark}[theorem]{Remark}

\numberwithin{equation}{section}
\numberwithin{theorem}{section}

 \newcommand{\F}{\mathbb{F}}
\newcommand{\K}{\mathbb{K}}
\newcommand{\D}[1]{D\(#1\)}
\def\scr{\scriptstyle}
\def\\{\cr}
\def\({\left(}
\def\){\right)}
\def\[{\left[}
\def\]{\right]}
\def\<{\langle}
\def\>{\rangle}
\def\fl#1{\left\lfloor#1\right\rfloor}
\def\rf#1{\left\lceil#1\right\rceil}
\def\le{\leqslant}
\def\ge{\geqslant}
\def\eps{\varepsilon}
\def\mand{\qquad\mbox{and}\qquad}

\def\vec#1{\mathbf{#1}}

\newcommand{\lcm}{\operatorname{lcm}}

\def\bl#1{\begin{color}{blue}#1\end{color}} % color text blue during edits

\newcommand{\C}{\mathbb{C}}
\newcommand{\Fq}{\mathbb{F}_q}
\newcommand{\Fp}{\mathbb{F}_p}
\newcommand{\Disc}[1]{\mathrm{Disc}\(#1\)}
\newcommand{\Res}[1]{\mathrm{Res}\(#1\)}
\newcommand{\ord}{\mathrm{ord}}

\newcommand{\Q}{\mathbb{Q}}
\newcommand{\Z}{\mathbb{Z}}
\renewcommand{\L}{\mathbb{L}}

\newcommand{\Norm}{\mathrm{Norm}}

%%%%%%%%%%%%%%%%%%%%%%%%%
% Alphabet calligraphie %
%%%%%%%%%%%%%%%%%%%%%%%%%
\def\cA{{\mathcal A}}
\def\cB{{\mathcal B}}
\def\cC{{\mathcal C}}
\def\cD{{\mathcal D}}
\def\cE{{\mathcal E}}
\def\cF{{\mathcal F}}
\def\cG{{\mathcal G}}
\def\cH{{\mathcal H}}
\def\cI{{\mathcal I}}
\def\cJ{{\mathcal J}}
\def\cK{{\mathcal K}}
\def\cL{{\mathcal L}}
\def\cM{{\mathcal M}}
\def\cN{{\mathcal N}}
\def\cO{{\mathcal O}}
\def\cP{{\mathcal P}}
\def\cQ{{\mathcal Q}}
\def\cR{{\mathcal R}}
\def\cS{{\mathcal S}}
\def\cT{{\mathcal T}}
\def\cU{{\mathcal U}}
\def\cV{{\mathcal V}}
\def\cW{{\mathcal W}}
\def\cX{{\mathcal X}}
\def\cY{{\mathcal Y}}
\def\cZ{{\mathcal Z}}

\def\fra{{\mathfrak a}} 
\def\frb{{\mathfrak b}}
\def\frc{{\mathfrak c}}
\def\frd{{\mathfrak d}}
\def\fre{{\mathfrak e}}
\def\frf{{\mathfrak f}}
\def\frg{{\mathfrak g}}
\def\frh{{\mathfrak h}}
\def\fri{{\mathfrak i}}
\def\frj{{\mathfrak j}}
\def\frk{{\mathfrak k}}
\def\frl{{\mathfrak l}}
\def\frm{{\mathfrak m}}
\def\frn{{\mathfrak n}}
\def\fro{{\mathfrak o}}
\def\frp{{\mathfrak p}}
\def\frq{{\mathfrak q}}
\def\frr{{\mathfrak r}}
\def\frs{{\mathfrak s}}
\def\frt{{\mathfrak t}}
\def\fru{{\mathfrak u}}
\def\frv{{\mathfrak v}}
\def\frw{{\mathfrak w}}
\def\frx{{\mathfrak x}}
\def\fry{{\mathfrak y}}
\def\frz{{\mathfrak z}}

\def\ov\Q{\overline{\Q}}
\def \brho{\boldsymbol{\rho}}

\def \fP {\mathfrak P}

\def \Prob{{\mathrm {}}}
\def\e{\mathbf{e}}
\def\ep{{\mathbf{\,e}}_p}
\def\epp{{\mathbf{\,e}}_{p^2}}
\def\em{{\mathbf{\,e}}_m}

\newcommand{\sR}{\ensuremath{\mathscr{R}}}
\newcommand{\sDI}{\ensuremath{\mathscr{DI}}}
\newcommand{\DI}{\ensuremath{\mathrm{DI}}}

\newcommand{\Orb}[1]{\mathrm{Orb}\(#1\)}
\newcommand{\aOrb}[1]{\overline{\mathrm{Orb}}\(#1\)}
\def \PrePer{{\mathrm{PrePer}}}
\def \Per{{\mathrm{Per}}}

\def \Nm{{\mathrm{Nm}}}
\def \Gal{{\mathrm{Gal}}}

\newenvironment{notation}[0]{%
  \begin{list}%
    {}%
    {\setlength{\itemindent}{0pt}
     \setlength{\labelwidth}{1\parindent}
     \setlength{\labelsep}{\parindent}
     \setlength{\leftmargin}{2\parindent}
     \setlength{\itemsep}{0pt}
     }%
   }%
  {\end{list}}

\definecolor{dgreen}{rgb}{0.,0.6,0.}
\def\tgreen#1{\begin{color}{dgreen}{\it{#1}}\end{color}}
\def\tblue#1{\begin{color}{blue}{\it{#1}}\end{color}}
\def\tred#1{\begin{color}{red}#1\end{color}}
\def\tmagenta#1{\begin{color}{magenta}{\it{#1}}\end{color}}
\def\tNavyBlue#1{\begin{color}{NavyBlue}{\it{#1}}\end{color}}
\def\tMaroon#1{\begin{color}{Maroon}{\it{#1}}\end{color}}

\title[Dynamical irreducibility modulo primes]{On the frequency of primes preserving dynamical irreducibility of polynomials}

\author[A. Ostafe] {Alina Ostafe}
\address{School of Mathematics and Statistics, University of New South Wales, Sydney NSW 2052, Australia}
\email{alina.ostafe@unsw.edu.au}

 \author[I.~E.~Shparlinski]{Igor E. Shparlinski}
 \address{I.E.S.: School of Mathematics and Statistics, University of New South Wales.
 Sydney, NSW 2052, Australia}
 \email{igor.shparlinski@unsw.edu.au}
 
 \begin{abstract} Towards a well-known open question in arithmetic dynamics, L.~M{\'e}rai, A.~Ostafe and I.~E.~Shparlinski (2023), 
have shown, for a class of polynomials $f \in \Z[X]$, 
 which in particular includes all quadratic polynomials,  that, under some natural conditions (necessary for quadratic polynomials),  the set of primes $p$, such that  all iterations of $f$ are irreducible 
 modulo $p$, is of relative density zero, with an explicit estimate on the rate 
 of decay. This result relies on some bounds on character sums via the 
 Brun sieve. Here we use the Selberg sieve and in some cases obtain a substantial quantitative
 improvement. 
 \end{abstract}

\pagenumbering{arabic}


\maketitle

\tableofcontents

\section{Introduction}
\subsection{Motivation}
\label{sec:mot} 

Given a polynomial $f\in \K[X]$ over a field $\K$,
we say that it is  {\it dynamically irreducible\/} if all iterates 
$$
  f^{(0)}(X)  = X, \qquad f^{(n)}(X)  = f\(f^{(n-1)}(X)\), \quad
  n =1, 2, \ldots\,.
$$
are irreducible. We note that this property is also known as {\it stability\/} and 
has been used in a number of works,  see, for example,~\cite{Ali,AyMcQ,Gomez_Nicolas,GNOS,Jon1,JB,Kon}. 
However, as in~\cite{MOS}, we prefer to use the more informative term of dynamical irreducibility, 
introduced by Heath-Brown and Micheli~\cite{H-BM}.

For a polynomial $f\in\Q[X]$ and a prime $p$ we denote by $\bar{f_p}\in\Fp[X]$ the  {\it image\/} of $f$ 
in the finite field $\F_p$ of $p$ elements, that is, the polynomial obtained by reducing modulo $p$ all the coefficients of $f$. 

Furthermore, we define
\begin{equation}
\label{eq:stable p}
\cP_f =  \{p:~\bar{f_p} \text{ is dynamically irreducible over } \Fp \}.
\end{equation}

Clearly, if $f\in\Q[X]$ is not dynamically irreducible then $\cP_f$ is a finite set
(consisting only of finitely many primes of ``bad reduction''). It has been asked in~\cite{BdMIJMST}
whether this is still true for any polynomial. 
More precisely, we recall  the    %% question, which has appeared as
following~\cite[Question~19.12]{BdMIJMST}.

\begin{question}
 Let $f\in\Q[X]$ be  of degree $d\geq 2$. Is it true that 
 $\cP_f$ 
 is a finite set?
\end{question}

Furthermore, 
Jones~\cite[Conjecture~6.3]{Jon2} has conjectured that $x^2+1$ is dynamically irreducible over $\Fp$ if and only if $p=3$. 

Ferraguti~\cite[Theorem~2.3]{Fer} gives a conditional result that if the size of the Galois group $\Gal\(f^{(n)}\)$  of $f^{(n)}$ is asymptotically close to its largest possible value then the set of  primes~\eqref{eq:stable p} has density zero. While this condition on 
the size of $\Gal\(f^{(n)}\)$  is very likely to be generically satisfied,  
 it may be difficult to verify it for concrete polynomials or find examples of such polynomials. 
 K{\"o}nig~\cite{Kon} has further restricted the class of polynomials for which the set  $\cP_f$ 
 is of positive density. 
%% is infinite. 
 
% We also recall that Jones~\cite[Conjecture~6.3]{Jon2} has conjectured that 
% $\cP_{X^2+1} = \{3\}$. 

Recently, it has been shown~\cite{MOS} that for  a special class of polynomials which includes trinomials of the form $f(X)=aX^d+bX^{d-1}+c\in \Z[X]$ of even degree, and hence all quadratic polynomials, the set $\cP_f$  is of relative zero density. 

More precisely,  for  $f \in \Z[X]$ and $Q\ge 2$, we define
\begin{equation}
\label{eq:PfQ}
 P_f(Q)=\#\{p\in [Q,   2Q] \cap \cP_f\}. 
\end{equation}

Then, by~\cite[Theorem~1.2]{MOS}, if the derivative $f'$  of  $f\in \Z[X]$ is
of the form
\begin{equation}\label{eq:f'shape}
f'(X)= g(X)^2 (aX+b),  \qquad  g(X) \in \Z[X], \ a,b \in \Z, \ a \ne 0,
\end{equation}
and $\gamma=-b/a$ is not  pre-periodic (that is, the set $\{f^{(n)}(\gamma):~n =1, 2, \ldots\}$ is infinite),  then 
\begin{equation}\label{eq:MOS bound}
 P_f(Q)  \le \frac{(\log \log\log\log Q)^{2+o(1)}} {\log\log\log Q}  \cdot \frac{Q }{\log Q}, \qquad  \text{as}\  Q\to \infty.
\end{equation}

It is important to emphasise that the above necessary conditions  are 
easily verifiable from the initial data.

Moreover,  it is also shown in~\cite[Theorem~1.4]{MOS} that 
under  the Generalised Riemann Hypothesis (GRH),   a stronger bound
\begin{equation}\label{eq:MOS bound GRH}
 P_f(Q)  \ll \frac{1} {\log\log Q}  \cdot \frac{Q }{\log Q} %%, \qquad  \text{as}\  Q\to \infty, 
\end{equation}
 holds; we refer to Section~\ref{sec:note} for some standard notations, such as
 the symbol `$\ll$'. 

The approach of~\cite{MOS} is based on a combination of several rather diverse
techniques: 
 \begin{itemize}
\item[(i)] effective results from Diophantine geometry~\cite{BEG}, 
\item[(ii)]  the square-sieve of Heath-Brown~\cite{H-B}, 
\item[(iii)]  a  refined bound of character sums over almost-primes~\cite{KonShp}. 
\end{itemize}

Here we use the Selberg sieve,  which makes a rather significant 
effect on our argument in Part~(iii) in the above.  This leads to a better version of~\eqref{eq:MOS bound}. 
 We also use a different input from Diophantine geometry (related to Part~(i)) and 
introduce another class of polynomials for which we establish quite substantial 
improvements of the bounds~\eqref{eq:MOS bound} and~\eqref{eq:MOS bound GRH}, 
roughly speaking with one less iteration of the logarithm in the saving against the trivial
bound. 

In fact,  the approach via  the Selberg sieve may also be used to refine the results of~\cite{KonShp}.
Namely, one can  obtain an asymptotic formula for the number of quadratic nonresidues modulo 
a prime $p$ in a very short interval
which holds for almost all primes $p$. The main point here is that the interval can be away from 
the origin (which makes it very different from classical results for initial intervals obtained
via the large sieve).

\subsection{Main results}
\label{sec:res}
The aforementioned bound  from~\cite{KonShp}  on character sums over almost-primes
is obtained via the classical Brun sieve,  see, for 
example,~\cite[Section~I.4.2, Theorem~3]{Ten}.

Here we consider a slightly different character sum, with so-called Selberg weights, 
see~\cite[Section~3.2]{MonVau}, 
and using this idea we obtain a stronger bound on $ P_f(Q)$. 

We define the following classes of polynomials in $\Z[X]$, which we denote by $\cP_1$ and $\cP_2$:
\begin{itemize}
\item[$\cP_1$:] $f \in \Z[X]$ of degree $d\ge 2$ 
such that the derivative $f'$ is of the form~\eqref{eq:f'shape}. 
In this case we define
\begin{equation}\label{eq: gamma-P1}
\gamma=-b/a
\end{equation}
to be the only critical point of odd multiplicity.

\item[$\cP_2$:]  $f \in \Z[X]$ of degree $d\ge 2$ of the form  
$$f(X)=f_dX^d+f_{d-1}X^{d-1}+f_0\in\Z[X], \quad f_d, f_{d-1}\ne 0,$$ such that $0$ is pre-periodic   under $f$. In this case we define 
\begin{equation}\label{eq: gamma-P2}
\gamma=-\frac{f_{d-1}(d-1)}{df_d}
\end{equation}
 to be the only non-zero critical point of $f$.
\end{itemize}

We first slightly improve~\eqref{eq:MOS bound} for polynomials $f \in  \cP_1$ and
extend it to  polynomials  $f \in  \cP_2$. 

%%We refer to Section~\ref{sec:note} for some standard notations, such as
%% the symbole `$\ll$'. 

\begin{theorem}
\label{thm:P1 P2}
Let $f \in  \cP_1 \cup \cP_2$. Assume that   
$\gamma$, defined for each of the classes $\cP_1$ and $\cP_2$ by~\eqref{eq: gamma-P1}
and~\eqref{eq: gamma-P2}, respectively, 
is not a pre-periodic point of $f$. Then   one has 
$$
 P_f(Q)  \ll \frac{1} {\log\log\log Q}  \cdot \frac{Q }{\log Q}.
$$
\end{theorem}

Obviously all quadratic polynomials are in the class $\cP_1$. 
Besides, it has been shown in~\cite[Section~1.2]{MOS} 
that polynomials of the form 
$$
f(X)=r(X-u)^d+s(X-u)^{d-1}+t\in \Z[X]
$$ 
with some $r,s,t,u \in \Z$, $r \ne 0$, of even degree $d$  such that 
$$\gamma=u-\frac{(d-1)s}{dr} 
$$
is not a pre-periodic point of $f$, belong also to the class $\cP_1$ 
(including the choice $s= 0$, that is, $f(X)=r(X-u)^d +t$).

Next we define the following sub-class of  polynomials from $\cP_1$, which we denote by $\cP_3$:
\begin{itemize}
\item[$\cP_3$:]  
 $f \in \Z[X]$ of degree $d\ge 2$ is such that $f \in \cP_1$ and 
 such that $0$ is pre-periodic under $f$.  
\end{itemize}  

For polynomials  $f \in \cP_3$ we obtain a much stronger bound than in~\eqref{eq:MOS bound}  and
in Theorem~\ref{thm:P1 P2}. 

\begin{theorem}
\label{thm:P3}
Let $f\in \cP_3$. Assume that   
$\gamma$, defined above for  the class $\cP_3$, is not a pre-periodic point of $f$. 
Then  one has 
$$
 P_f(Q)  \ll  \frac{1} {\log\log Q}  \cdot \frac{Q }{\log Q}.
$$
\end{theorem}

To show that the class $\cP_3$ contains infinitely many dynamically irreducible polynomials, 
we note that for $f(X) = aX^d - acX^{d-1} + c \in \Z[X]$ we have $f(0) = c$ and $f(c) = c$. 
Hence $0$ is pre-periodic for $f$, thus $f \in \cP_3$. 
Now, for any integers $a$ and $c$ such that for some prime $p$ we have 
$$
p\nmid a, \qquad p \mid c,  \qquad p^2 \nmid c,
$$ 
one can easily verify that by the Eisenstein criterion $f$ is dynamically irreducible, see~\cite[Lemma~1.3]{Odoni}. 
Indeed, easy induction shows that  for any $n\ge 1$, the leading coefficient of $f^{(n)}$ 
is a power of $a$, all other coefficients are multiples of $c$ and 
$$
f^{(n)}(0) =c.
$$ 
We now give a conditional (on the GRH) estimate which  for polynomials  $f \in \cP_3$ 
is much stronger than~\eqref{eq:MOS bound GRH}.

\begin{theorem}
\label{thm:P3-GRH}
Let $f\in \cP_3$. Assume that   
$\gamma$, defined above for  the class $\cP_3$, is not a pre-periodic point of $f$. Then, assuming the GRH, one has
$$
 P_f(Q)  \ll \frac{Q }{(\log Q)^2}.
$$
%where the implied constant depends only on $f$. 
\end{theorem}

\section{Preliminaries}

 \subsection{Notation, general  conventions and definitions}
\label{sec:note}

We use the Landau symbol $O$ and the Vinogradov symbol $\ll$. Recall that the
assertions $U=O(V)$ and $U \ll V$  are both equivalent to the inequality $|U|\le cV$ 
with some (effective) constant $c>0$ that may depend on the involved polynomial $f$.

Throughout the paper, $p$ always denotes a prime number. 
When we say that $u\in\Z$  is a square we mean that $u$ is a perfect  integer square.

Given  polynomials  $g,h\in\Q[X]$ we use $\Res{g,h}$ to denote their resultant.

We define the {\it Weil logarithmic height\/} of $a/b\in\Q$ as $$h(a/b)=\max \{ \log |a|,\log|b|\},$$
with the convention $h(0)=0$.

We always assume that $Q$ is large enough so that various iterated logarithms are 
well-defined.

\subsection{Jacobi symbols and their sums}
For $n\geq 3$, $\left(\frac{\cdot}{n}\right)$ denotes the {\it Jacobi symbol\/}, which is identical to the {\it Legendre symbol\/}  if $n$ is prime. We recall the following well-known properties, see~\cite[Section~3.5]{IwKow}.

\begin{lemma}\label{lemma:reciprocity}
 For odd integers $m,n\geq 3$ we have
$$
  \left(\frac{m}{n}\right)=(-1)^{\frac{m-1}{2}\frac{n-1}{2}}\left(\frac{n}{m}\right)
  \mand
  \left(\frac{2}{n}\right)=(-1)^{\frac{n^2-1}{8}}.
  $$
\end{lemma}

Next, we need the following modification of the classical P\'olya--Vinogradov bound of character sums.

\begin{lemma}\label{lem:P-V}
For any integers $e, K\ge 1$, and an odd integer $v\geq 3$ which is  not a   square, we have
$$
\sum_{\substack{k =1\\ e\mid k}}^K \left(\frac{k}{v}\right) \ll  v^{1/2}  \log v.
$$
\end{lemma}

\begin{proof} If $\gcd(e,v)>1$ then all Jacobi symbols are vanishing.
Otherwise using the multiplicativity of the Jacobi symbol we write 
$$
\sum_{\substack{k =1\\ e\mid k}}^K \left(\frac{k}{v}\right)
=   \left(\frac{e}{v}\right)\sum_{1 \le m \le K/e}  \left(\frac{m}{v}\right)
$$
and the result follows from  the classical P\'olya--Vinogradov bound, 
see, for example,~\cite[Theorem~12.5]{IwKow} or~\cite[Theorem~9.18]{MonVau}. 
\end{proof}

We also recall that 
under the GRH we have a rather strong bound for sums over primes, 
see~\cite[Equation~(13.21)]{MonVau}.

\begin{lemma}\label{lemma:char sum-GRH}
For any   positive integers $q$ and $M$, where $q\geq 2$ is not a perfect square, we have
$$
\left|\sum_{p \le M}\left(\frac{p}{q}\right) \right|   \ll M^{1/2} \log (qM). 
$$
\end{lemma}

 \subsection{Squares modulo $\cS$-units}
 
Given a finite set of non-Archimedean valuations $\cS$ of a number field $\K$
we say that that $\alpha \in \K^*$ is an $\cS$-integer if $\nu (\alpha) \ge 0$ 
for any valuation $\nu \not \in \cS$. 
Furthermore, we say that $u \in \K^*$ is an $\cS$-unit if $\nu (u) = 0$ 
for any valuation $\nu \not \in \cS$. 

\begin{lemma}\label{lem:GCD}
Let  $\cS$ be  a finite set of non-Archimedean valuations of a number field $\K$.
Let $\alpha, \beta, \vartheta\in  \K^*$ be $\cS$-integers and let $s_1, s_2 \in \K^*$ be 
 $\cS$-units  such that 
$$
\alpha (\alpha \beta + s_1) = s_2 \vartheta^2.
$$
Then $\alpha = s_0  a^2$ for some $\cS$-unit  $s_0 \in \K^*$ and an algebraic 
integer $a \in \K^*$. 
\end{lemma}

\begin{proof}It is easy to see that for any valuation $\nu \not \in \cS$, if 
$\nu(\alpha) > 0$ then  $\nu(\vartheta)>0$. Hence writing the above relation 
as $s_1  \alpha = s_2 \vartheta^2 - \alpha^2 \beta$, and recalling that $\nu(\beta)\ge 0$, we see that $\nu(\alpha)$ is 
even. The result now follows. 
\end{proof}

We also need the following effective result of B{\'e}rczes,   Evertse and Gy\H{o}ry~\cite[Theorem~2.2]{BEG}, which bounds the height  of $\cS$-integer solutions to a hyperelliptic equation. We present it in the form needed for the proof of Theorems~\ref{thm:P1 P2} and~\ref{thm:P3}. In particular, we formulate
it only over $\Q$, while the original result applies to arbitrary number fields.
 
Let $\cS$  be a finite set of primes of cardinality $s=\#\cS$ and define
$\Z_\cS$ to be the ring of $\cS$-integers, that is, the set of rational numbers $r$  with $v_p(r)\ge 0$ for any $p\not\in \cS$. 
Put
$$
Q_\cS=\prod_{p\in \cS} p.
$$

\begin{lemma}\label{lemma:BEG}
Let $f\in \Z_\cS[X]$ be a polynomial of degree $d\geq 3$ without multiple zeros, and let $b\in\Z_\cS$ 
be a nonzero $\cS$-integer. 
If $x,y\in \Z_\cS$ are solutions to the equation
$$
f(x)=by^2,
$$
then
$$
h(x),h(y)\leq (4d(s+1))^{212d^4 (s+1)} Q_S^{20 d^3} \exp(O(h(b))).
$$
\end{lemma}

We remark that we use $s+1$ in Lemma~\ref{lemma:BEG} rather than $s$ (as in ~\cite[Theorem~2.2]{BEG}) since we have not 
included in $\cS$ the unique Archimedean valuation.

\section{Square products in some products with iterations}

\subsection{Resultants with iterations}
%%\label{sec:res-iter}
We first recall the following necessary condition for  a polynomial 
to be dynamically irreducible over a finite field of 
odd characteristic~\cite[Corollary~3.3]{GNOS}.

\begin{lemma}
\label{lemma:stab}
Let $q$ be an odd prime power, and let $f\in\F_q[X]$ be a dynamically irreducible polynomial 
 of degree $d\geq  2$ with leading coefficient $f_d$, non-constant derivative 
 $f'$, $\deg f'=k\le d-1$. 
  Then the following properties hold: 
 \begin{enumerate}
 \item if $d$ is even, then $\Disc{f}$ and $f_d^k \Res{f^{(n)},f'}$, $n\geq 2$, are non-squares in $\Fq$,  
 \item if $d$ is odd, then $\Disc{f}$ and $(-1)^{\frac{d-1}{2}}f_d^{(n-1)k+1} \Res{f^{(n)},f'}$, $n\geq 2$, are squares in $\Fq$. 
   \end{enumerate}
\end{lemma}

Given a dynamically irreducible polynomial $f \in \Z[X]$ of degree $d\ge 2$ with the leading 
coefficient $f_d \ne 0$, we  write 
\begin{equation}
\label{eq:def un}
f_d \cdot \Res{f^{(n)},f'}=2^{\nu_n}u_n, \quad n\geq 2,
\end{equation}
with some odd integers $u_n$. 
We observe that  the dynamical  irreducibility of  $f$ implies that $\Res{f^{(n)},f'} \ne0$,
and hence the above parameters $\nu_n$ and $u_n$, are always well-defined. 

We also recall the following estimate on the size of iterates, see~\cite[Lemma~2.8]{MOS}. 

\begin{lemma}
\label{lem:res height}
Let $f\in\Q[X]$ be a polynomial of degree $d\ge 1$. Then, for any $n\ge 1$, we have
$$
h\(\Res{f^{(n)},f'}\)\ll d^n.
$$
\end{lemma}

\subsection{Set-up of sieving parameters}
\label{sec:par}
Let us fix some integer parameters $N$ and $t$ to be chosen later.

As in~\cite{MOS} we observe that 
by the Dirichlet principle there is a set $\cN\subseteq [N,N+t]$ of size
\begin{equation}
\label{eq:large N}
\#\cN\geq \frac{1}{4}t
\end{equation}
 such that for all  $r,s\in\cN$ we have
%\begin{equation}\label{eq:parity}
$$
u_r\equiv u_s \mod 4  \mand  \nu_r\equiv \nu_s \mod 2.
$$
%\end{equation}
Therefore, since $u_n$, $n\ge 2$, are odd, we have
\begin{equation}
\label{eq:u nu}
u_r+u_s \equiv 2 \mod 4 \mand \nu_r+\nu_s\equiv 0 \mod 2.
\end{equation}

Using that for an odd $m$ we have $2 \mid m-1$ and $8 \mid m^2-1$, we conclude that 
\begin{equation}\label{eq:parity}
(-1)^{\frac{u_{r}+u_{s}-2}{2}\frac{m-1}{2}+(\nu_{r}+\nu_{s} ) \frac{m^2-1}{8}} =1.
 \end{equation}

\subsection{Classes $\cP_1$ and $\cP_2$}

One of our main ingredients is the following result, essentially 
established in~\cite[Section~3.2]{MOS} for the class of polynomials $\cP_1$. 

\begin{lemma}\label{lem:squaresP1P2}
Let $f \in \cP_1 \cup \cP_2$ be dynamically irreducible.  Assume that   
$\gamma$ defined for each of the classes $\cP_1$ and $\cP_2$ by~\eqref{eq: gamma-P1}
and~\eqref{eq: gamma-P2}, respectively,  is not a pre-periodic point of $f$. 
Then there is a constant $c_1> 0$, depending only on $f$, such that for 
$t \le c_1 \log N$ the products $|u_m u_n|$ with $m,n\in\cN$, $m<n$, never form a  square. 
\end{lemma}

\begin{proof} The proof follows  similar argument as in~\cite[Section~3.2]{MOS}, which already proves the statement for the polynomials $f\in \cP_1$. For convenience, we repeat some details, to prove the statement for the class of polynomials $\cP_2$ as well.

Let $f\in \cP_2$ and let $m<n$ be a pair of integers in $\cN$ such that $u_{m}u_{n}$ is a  square. 
 Then, as in~\cite[Section~3.2]{MOS}, using properties of resultants,  
we obtain, recalling~\eqref{eq:u nu}, that the resultant $\Res{f^{(m)}f^{(n)},f'} $
is also a square.

Now, since $f\in \cP_2$, and using again properties of resultants as in~\cite[Section~3.2]{MOS}, we obtain that 
\begin{equation}
\label{eq:res mn}
\begin{split}
\Res{f^{(m)}f^{(n)},f'}&=(-1)^{(d-1)(d^{m}+d^{n})}(df_d )^{d^{m}+d^{n}}\\
&\qquad \quad  \times \(f^{(m)}(0) f^{(n)}(0)\)^{d-2} f^{(m)}(\gamma)f^{(n)}(\gamma).
\end{split}
\end{equation}

Now, since $f$ is dynamically irreducible, $0$ is not periodic for $f$ and thus  
$$
f^{(m)}(0), f^{(n)}(0)\ne 0.
$$
Moreover, since $0$ is pre-periodic under $f$,  the orbit of $0$,
$$
\cO_f(0)=\{f^{(k)}(0): k\ge 1\},
$$
is a finite set.

We let $\cS$ be the (finite) set of prime divisors of $d$, $f_d$ and the finitely many elements in $\cO_f(0)$. 

From~\eqref{eq:res mn}, we conclude that 
$$
\alpha f^{(n-m)}(\alpha) =u\beta^2,
$$
where $\alpha=f^{(m)}(\gamma)$ and $\beta$ are $\cS$-integers in $\Q$, $u$ is an $\cS$-unit in $\Q$ and $\deg f^{(n-m)}\le d^t$.

Now the proof is concluded exactly as in~\cite[Section~3.2]{MOS}. 

Indeed, 
since $f$ is dynamically irreducible of degree at least two,  $Xf^{(n-m)}(X)$ is a polynomial of degree at least $3$ without multiple roots in $\ov\Q$. We can apply now Lemma~\ref{lemma:BEG} with the polynomial $Xf^{(n-m)}(X)$ to conclude that 
\begin{equation}
\label{eq:upper}
h(\alpha)\le \exp(O(d^{4t})).
\end{equation}
On the other hand, since $\gamma$ is not a pre-periodic point of the polynomial  $f$, as in~\cite[Section~3.2]{MOS}, there exists a positive integer $n_0$ depending only on $f$ such that $h\(f^{(n_0)}(\gamma)\)\ge c_0+1$, where $c_0$ is a constant depending only on $f$, and thus we have
$$
h(\alpha)=h\(f^{(n-n_0)}\(f^{(n_0)}(\gamma)\)\)\ge d^{n-n_0}\(h\(f^{(n_0)}(\gamma)\)-c_0\) \ge d^{N-n_0}.
$$
We choose now a suitable constant $c_1$, depending only on $f$, and also a large enough $N$ 
  to obtain a contradiction 
with~\eqref{eq:upper}. We thus deduce that there is no nontrivial pair $m,n\in\cN$, $m<n$, such that  
$u_{m}u_{n}$
is a square in $\Q$, which concludes the proof.
\end{proof}

\begin{remark} 
We note that for the proof of Lemma~\ref{lem:squaresP1P2}, one does not need the condition that $f$ is dynamically irreducible, but only that $0$ is not periodic under $f$ and that the roots of $Xf^{(k)}(X)$, $k\ge 1$, are all distinct. In fact the latter condition can also be relaxed via~\cite[Theorem~2.1]{BBGMOS}.
\end{remark}

\subsection{Class $\cP_3$}
We now show that for polynomials  $f \in \cP_3$
the products  $|u_m u_n|$  form a   square
only in the trivial case $m=n$, provided $m$ and $n$ are large enough. 

\begin{lemma}\label{lem:squaresP3}
Let $f \in \cP_3$ be dynamically irreducible.  Assume that   
$\gamma$ defined for  f in the class $\cP_3$ by~\eqref{eq: gamma-P1}
 is not a pre-periodic point of $f$. 
Then there is a constant $c_2$ , depending only on $f$, such that for 
$m > c_2$ the products $|u_m u_n|$ with $m,n\in\cN$, $m<n$, 
never form a square. 
\end{lemma}

\begin{proof} 
We see that 
\begin{equation}
\label{eq:fn fm}
f^{(n)}(X) = f^{(m)}(X) G_{m,n}(X) + f^{(n)}(0) .
\end{equation}

Clearly $ f^{(n)}(0) \ne 0$ for $n \ge 1$ as otherwise $f^{(n)}(X) $ is reducible.
By pre-periodicity of $0$, the values $f^{(n)}(0)$, $n =1, 2, \ldots$, belong to 
a finite set $\cT$. Defining $\cS$ as in the proof of Lemma~\ref{lem:squaresP1P2}, 
including also the  the primes dividing the finitely many elements in $\cT$, we see, 
using properties of resultants as in~\cite[Section~3.2]{MOS}, and~\eqref{eq:def un}
and~\eqref{eq:fn fm},
that if $|u_m u_n|$ is a square then we have 
$$
f^{(m)}\(\gamma\) \( f^{(m)}\(\gamma\) \beta+ s_1\) = s_2 \vartheta^2
$$
for some $S$-integers $\beta, \vartheta\in \Q^*$ and 
 $\cS$-units $s_1, s_2 \in \Q^*$. Hence, by Lemma~\ref{lem:GCD}
\begin{equation}
\label{eq:fm a2}
f^{(m)}\(\gamma\) =  s_0  a^2, 
\end{equation}
 for some $\cS$-unit  $s_0 \in \Q^*$ and an 
integer $a \in \Z$. We also note that since $\gamma$ is not pre-periodic, it must be nonzero.

Assuming that $m\ge 2$, we can write~\eqref{eq:fm a2} as 
$$
f^{(2)}\(f^{(m-2)}\(\gamma\) \)=  s_0  a^2.
$$
Since, by our assumption, $f^{(2)}(X)$ is irreducible  by  Lemma~\ref{lemma:BEG} 
we see that the height of $f^{(m-2)}\(\gamma\)$ is bounded only in terms of $f$, and thus~\eqref{eq:fm a2} 
is possible for finitely many values of $f^{(m-2)}\(\gamma\)$. Since $\gamma$ is
not pre-periodic, this implies there is a constant $c_2$, depending only on $f$, such that the possible values of $m$ satisfy $m\le c_2$. 
This completes the proof.
\end{proof}

\section{Proofs of main results}

\subsection{Reduction to bounds of character sums}
 
We recall the definition of the elements $u_n$ and of the set $\cN$ in Section~\ref{sec:par}. We split $\cN$ into two subsets  of $n\in \cN$ for which   $u_n$ 
has a prescribed sign. Let $\cM$ be the largest subset and thus 
(also recalling~\eqref{eq:large N}) we see that
we have  
$$
u_m u_n > 0, \quad m,n \in   \cM, \mand 
\#\cM \ge  \frac{1}{2} \# \cN \ge \frac{1}{8}t. 
$$

Let us consider the sum
\begin{equation}
\label{eq:Sum S}
S=\sum_{Q \le p \le 2Q} \(\sum_{n\in\cM} \left(\frac{f_d  \cdot \Res{f^{(n)},f'}}{p} \right)\)^2.
\end{equation} 
If $\bar{f_p}$ is dynamically irreducible modulo $p$, then by Lemma~\ref{lemma:stab}, we have
\begin{align*}
\left| \sum_{m\in\cM} \left(\frac{f_d \cdot \Res{f^{(m)},f'}}{p} \right)\right|& =
\left|  \sum_{n\in\cM} \left(\frac{f_d ^{d-1}\cdot \Res{f^{(n)},f'}}{p} \right)\right|\\
& = \#\cM\leq \frac{t}{8},
\end{align*}
and  thus
\begin{equation}
\label{eq:Pf S}
P_f(Q)\leq 64 \frac{S}{t^2},
\end{equation}
where $P_f(Q)$ is defined by~\eqref{eq:PfQ} (which is a full analogue 
of~\cite[Equation~(3.4)]{MOS}).

Finally,  using Lemma~\ref{lemma:reciprocity}, we obtain 
\begin{equation}
\label{eq: Spm flip}
S = \sum_{p \in [Q,2Q]} \(\sum_{n\in\cM } (-1)^{\frac{u_n-1}{2}\frac{p-1}{2}+\nu_n \frac{p^2-1}{8}}\left(\frac{p}{u_n} \right)\)^2.
\end{equation} 

\subsection{Selberg weights}
We appeal to~\cite[Section~3.2]{MonVau}. In particular, we choose some 
parameter $z \ge 2$ and set 
\begin{equation}
\label{eq: PrimeProd}
\fP = \prod_{p \le z} p.
\end{equation}
 It is shown in~\cite[Section~3.2]{MonVau}, that there are some   real numbers $\lambda^+_n$, known as {\it Selberg weights\/},  such that  
\begin{itemize}
\item by construction of the weights $\Lambda_r$ 
 in~\cite[Section~3.2]{MonVau} (which are denoted as $\Lambda_d$ in~\cite{MonVau} 
 but here $d$ is already reserved for the degree of $f$) : for
$n \ge z^2$,
\begin{equation}
\label{eq:lambda_n = 0}
 \lambda^+_n  = 0 ; 
\end{equation}
\item by~\cite[Equations~(3.17) and~(3.22)]{MonVau} combined 
with~\cite[Exercise~2.1.17]{MonVau}:
\begin{equation}\label{eq: sum|lambda_n|}
 \sum_{n=1}^\infty |\lambda^+_n| \ll \frac{z^2}{(\log z)^2};
\end{equation}
 \item by the definition of $\lambda^+_n$ and~\cite[Equations~(3.12) and~(3.13)]{MonVau}: 
\begin{equation}\label{eq: sum lambda_e e|q}
 \sum_{e\mid q} \lambda^+_e \ge \begin{cases} 1,  & \text{if\ } \gcd \(q,\fP\) = 1, \\
 0, & \text{otherwise,}
 \end{cases} 
\end{equation}
where $\fP$ is defined by~\eqref{eq: PrimeProd};
 \item by the argument given in the proof of~\cite[Theorem~3.2]{MonVau}, with the above choice of $\fP$ and 
 real weights $\Lambda_r$ 
  as in~\cite[Section~3.2]{MonVau}:  for any $Z$ with  $z^3 \le Z \le z^{O(1)}$, 
 \begin{equation}
\label{eq: sum lambda_n dyadic}
\begin{split}
 \sum_{q \in [Z,2Z]}  \sum_{e\mid q} \lambda^+_e& =
Z  \sum_{e \le z^2}\frac{\lambda^+_e}{e} + O(z^2) \\
 & = Z  \sum_{e \le z^2} \sum_{\substack{r,s\le z\\ \lcm[r,s] = e}}\Lambda_r \Lambda_s  + O(z^2) 
   \ll \frac{Z}{\log Z},
  \end{split} 
\end{equation}
for the optimal choice of the weights   $\Lambda_r$, given 
by~\cite[Equations~(3.12) and~(3.15)]{MonVau}. 
\end{itemize}

We also note that these properties are also summarised by Harper~\cite[Section~5]{Harp}. 
In fact our approach follows the argument in~\cite[Section~5]{Harp}, however modified to
accommodate our sum $S$.

\subsection{Estimating the  sum $S$}
Next, given the above weights  $\lambda^+_n$, with 
 \begin{equation}
\label{eq: z}
z = Q^{1/4},
\end{equation} 
and with $\fP$ as in~\eqref{eq: PrimeProd},  
using~\eqref{eq: Spm flip} and then~\eqref{eq: sum lambda_e e|q},  
we write 
\begin{align*}
S & \le \sum_{\substack{q \in [Q,2Q]\\\gcd(q,\fP) = 1}}  \(\sum_{n\in\cM } (-1)^{\frac{u_n-1}{2}\frac{p-1}{2}+\nu_n \frac{p^2-1}{8}}\left(\frac{p}{u_n} \right)\)^2\\
& =   \sum_{q\in[Q,2Q]}  \sum_{e\mid q} \lambda^+_e \(\sum_{n\in\cM } (-1)^{\frac{u_n-1}{2}\frac{q-1}{2}+\nu_n \frac{q^2-1}{8}}\left(\frac{q}{u_n} \right)\)^2\\
&\le 
 \sum_{q\in[Q,2Q]}  \sum_{e\mid q} \lambda^+_e \sum_{n_1,n_2\in\cM} 
(-1)^{\frac{u_{n_1}+u_{n_2}-2}{2}\frac{q-1}{2}+(\nu_{n_1}+\nu_{n_2} ) \frac{m^2-1}{8}}\left(\frac{q}{u_{n_1}u_{n_2}} \right). 
\end{align*}

Recalling~\eqref{eq:u nu}, we conclude that 
$$
S  \le \sum_{q\in[Q,2Q]}  \sum_{e\mid q} \lambda^+_e 
\sum_{n_1,n_2\in\cM}  \left(\frac{q}{u_{n_1}u_{n_2}}\right).
$$
By the positivity condition~\eqref{eq: sum lambda_e e|q} and then by the inequality~\eqref{eq: sum lambda_n dyadic}, the contribution $D$ from the diagonal 
terms is 
$$
D \le \# \cM \sum_{q\in[Q,2Q]}  \sum_{e\mid q} \lambda^+_e 
\le   \# \cM \frac{Q}{\log Q}. 
$$
Hence, 
\begin{equation}
\label{eq:S and T}
S \ll t  \frac{Q}{\log Q} + T,
\end{equation}  
where, changing the order of summation twice (and using~\eqref{eq:lambda_n = 0})
we write   
 \begin{equation}
\label{eq:  sum T}
\begin{split}
T & =
 \sum_{q\in[Q,2Q]}  \sum_{e\mid q} \lambda^+_e 
  \sum_{\substack{n_1,n_2\in\cM\\n_1\ne n_2}}  
\left(\frac{q}{u_{n_1}u_{n_2}}\right)\\
& = \sum_{\substack{n_1,n_2\in\cM\\n_1\ne n_2}}   \sum_{e \le z^2} \lambda^+_e   
  \sum_{\substack{q\in[Q,2Q]\\ e \mid q}}  
\left(\frac{q}{u_{n_1}u_{n_2}}\right).
  \end{split} 
\end{equation}

\subsection{Proof of Theorem~\ref{thm:P1 P2}}

Applying now Lemma~\ref{lem:squaresP1P2} with 
 some  sufficiently small constant $c_1>0$, 
depending only on $f$ and 
\begin{equation}
\label{eq:N t 12}
N = \rf{c_3 \log \log Q} \mand t= \fl{c_4\log\log\log Q},
\end{equation} 
we conclude that the products  $u_{n_1}u_{n_2}$,  $n_1,n_2\in\cM$, $n_1\ne n_2$, are 
never a   square and, by Lemma~\ref{lem:res height}, one has 
$$
3  \le |u_{n_1}u_{n_2}| \le Q^{1/2}.
$$
(provided that $Q$ is large enough).

Hence for $T$ in~\eqref{eq:  sum T}, using Lemma~\ref{lem:P-V},  we derive 
$$
T  \ll t^2  Q^{1/4}\log Q \sum_{e \le z^2} |\lambda^+_e| .
$$ 
Recalling~\eqref{eq: sum|lambda_n|} and the choice of $z$ in~\eqref{eq: z}, 
we now derive 
$$
T  \ll t^2  \frac{z^2}{(\log z)^2}  Q^{1/4}  \log Q \ll t^2 Q^{3/4}(\log Q)^{-1}.
$$ 
Substituting this bound in~\eqref{eq:S and T}, we obtain
$$
S \ll t  \frac{Q}{\log Q} + t^2 Q^{3/4}(\log Q)^{-1}.
$$ 
Thus we see from~\eqref{eq:Pf S} that
 \begin{equation}
\label{eq: PfQ-Bound}
P_f(Q) \ll t^{-1}   \frac{Q}{\log Q}  + Q^{3/4}(\log Q)^{-1},
\end{equation} 
and with  the choice of $t$ in~\eqref{eq:N t 12} the result follows.

\subsection{Proof of Theorem~\ref{thm:P3}}  
Applying now Lemma~\ref{lem:squaresP3} with 
 some  sufficiently small constant $c_2>0$, 
depending only on $f$ and 
\begin{equation}
\label{eq:N t 3}
N = \rf{c_5 \log \log Q} \mand t=  2N
\end{equation}
for some sufficiently small $c_5$, depending only on $f$, 
we conclude that the products  $u_{n_1}u_{n_2}$,  $n_1,n_2\in\cM$, $n_1\ne n_2$, are 
never  squares and, by Lemma~\ref{lem:res height}, one has
$$
3  \le |u_{n_1}u_{n_2}| \le Q^{1/2}.
$$ 
Hence  we still have the bound~\eqref{eq: PfQ-Bound}
however  with  the choice of $t$ in~\eqref{eq:N t 3}. 

\subsection{Proof of Theorem~\ref{thm:P3-GRH}} 
As in the proof of Theorem~\ref{thm:P3}, we choose 
\begin{equation}
\label{eq:N t 3 GRH}
N = \rf{c_6  \log Q} \mand t=  2N
\end{equation}
for some sufficiently small $c_6$, depending only on $f$. Therefore, by Lemma~\ref{lem:squaresP3}, the products  $u_{n_1}u_{n_2}$,  $n_1,n_2\in\cM$, $n_1\ne n_2$, are 
never  squares and, by Lemma~\ref{lem:res height}, one has
$$
3  \le |u_{n_1}u_{n_2}| \le \exp\(O\(d^{N+t}\)\)\le \exp\(Q^{1/3}\)
$$
if $c_6$ is small enough.

We recall the inequality~\eqref{eq:Pf S}, however this time we do not  expand the sum over primes in $S$, that is, by~\eqref{eq:parity} and~\eqref{eq: Spm flip} we have
$$
S  \le \sum_{p\in[Q,2Q]}
\sum_{n_1,n_2\in\cM}  \left(\frac{p}{u_{n_1}u_{n_2}}\right).
$$
Hence, using  Lemma~\ref{lemma:char sum-GRH} for the non-diagonal term, we arrive 
(recalling the choice of $N$ and $t$ in~\eqref{eq:N t 3 GRH}) to 
$$
P_f(Q)\ll  t^{-1} \frac{Q}{  \log Q } +  Q^{5/6}\ll \frac{Q}{(\log Q)^2},
$$
which concludes the proof.

 \section*{Acknowledgements} 
 
The authors  are very grateful to Adam Harper for his suggestion to use the Selberg 
sieve in order to estimate the character sums~\eqref{eq:Sum S} which is crucial 
for our result. The authors also would like to thank Bryce Kerr for discussing  the
potential use
of the Selberg sieve for a  
refinement of the result of~\cite[Theorem~2]{KonShp}, see Section~\ref{sec:mot}.

  During the preparation of this paper,  the authors were  supported by  Australian Research Council Grants DP200100355 and  DP230100530.
  
   The first author  gratefully acknowledges the hospitality and support of the Max Planck Institute for Mathematics in Bonn  and 
  the Mittag-Leffler Institute in Stockholm, where parts of this work has been carried out.
 The second author was supported by the Knut and Alice Wallenberg Fellowship and would also like to thank the Mittag-Leffler Institute for its hospitality 
and excellent working environment.

\end{document}